\documentclass[oneside,12pt]{amsart}

\usepackage{amscd,amsmath,latexsym,bm,pdfpages,verbatim,amsthm}
\usepackage[mathcal]{euscript}
\usepackage{graphicx}
\usepackage{amssymb}          
\usepackage{epstopdf}
\usepackage[all]{xy}

\newtheorem{thm}{Theorem}[section]
\newtheorem{cor}[thm]{Corollary}
\newtheorem{lem}[thm]{Lemma}
\newtheorem{prop}[thm]{Proposition}
\newtheorem*{thma}{Theorem 7.5}
\newtheorem*{cora}{Corollary 7.6}
\newtheorem*{convention}{Convention}

\theoremstyle{definition}
\newtheorem{defin}[thm]{Definition}
\theoremstyle{definition}
\newtheorem{construction}[thm]{Construction}

\theoremstyle{definition}
\newtheorem{exm}[thm]{Example}

\newtheorem{remark}[thm]{Remark}

\theoremstyle{remark}
\newtheorem*{rem}{Remark}

\newcommand{\R}{{\mathbb R}}
\newcommand{\C}{{\mathbb C}}

\begin{document}
\def\X#1#2{r(v^{#2}\ds{\prod_{i \in #1}}{x_{i}})}
\def\skp#1{\vskip#1cm\relax}
\def\C{{\mathbb C}}
\def\R{{\mathbb R}}
\def\ce{{\mathbb C}}
\def\erre{{\mathbb R}}
\def\efe{{\mathbb F}}
\def\ene{{\mathbb N}}
\def\UNO{1\mkern-7mu1}
\def\ee{{\mathbb E}}
\def\pee{{\mathbb P}}
\def\ene{{\mathbb N}}
\def\tn{{\bf T$^{n}$}}
\def\pn{$P^{n}\/$}
\def\cn{{\bf C$^{n}$}}
\def\z2{{\bf Z}$_{2}$}
\def\zl2{{\bf Z}$_{(2)}$}
\def\block{\rule{2.4mm}{2.4mm}}
\def\rplus{{\bf R$_{+}}}
\def\nd{\noindent}
\def\becomes{\colon\hspace{-2,5mm}=}
\def\ds{\displaystyle}
\def\red{\color{red}}
\def\black{\color{black}}
\def\s{\sigma}
\numberwithin{equation}{section}
\title[Operations on polyhedral products   \ldots]{ Operations on polyhedral products and a new 
topological construction of infinite families of toric manifolds}

\skp{0.2}

\author[A.~Bahri]{A.~Bahri}
\address{Department of Mathematics,
Rider University, Lawrenceville, NJ 08648, U.S.A.}
\email{bahri@rider.edu}

\author[M.~Bendersky]{M.~Bendersky}
\address{Department of Mathematics
CUNY,  East 695 Park Avenue New York, NY 10065, U.S.A.}
\email{mbenders@hunter.cuny.edu}

\author[F.~R.~Cohen]{F.~R.~Cohen}
\address{Department of Mathematics,
University of Rochester, Rochester, NY 14625, U.S.A.}
\email{cohf@math.rochester.edu}

\author[S.~Gitler]{S.~Gitler}
\address{}



\begin{abstract}

A combinatorial construction  is used to analyze the properties of polyhedral products  
\cite{bbcgpnas}  and generalized moment-angle complexes with respect to 
certain operations on CW pairs including exponentiation. 
This allows for the construction of infinite families of toric manifolds,  associated to a given one, 
in a way which simplifies the combinatorial input and consequently, the presentation of the cohomology 
rings.  The new input is the interaction of a purely combinatorial construction
with natural associated geometric constructions related to polyhedral products
and toric manifolds. Applications of the methods and results developed here have appeared in 
\cite{ustinovsky1}, \cite{ustinovsky2}, \cite{erokhovets}, \cite{gt}, \cite{choi.park}, \cite{suciu} and 
\cite{gptw}. \\[0.1mm]
\begin{center}\normalsize{{ \em This paper is dedicated to the memory of Samuel Gitler.}}
\end{center}
\end{abstract}
\maketitle
\tableofcontents
\section{Introduction}\label{sec:introduction}
 The polyhedral product  $Z(K;(\underline{X}, \underline{A}))$ is a CW-complex valued
functor of two variables: the first, an abstract simplicial complex $K$ on $m$ vertices and the
second, a family of (based) CW pairs
$$(\underline{X}, \underline{A}) \;=\; \{(X_1,A_1), (X_2,A_2), \ldots ,(X_m,A_m)\}.$$

\nd It is defined as a union of products inside $\prod\limits_{i=1}^{m}{X_i}\;$ each parameterized by a simplex 
in $K$.
\skp{0.05}
Polyhedral products generalize the spaces called {\em moment-angle complexes\/} which were
developed first  by Buchstaber and Panov in  \cite{buchstaber.panov.1} and correspond to the case
$(X_{i},A_{i}) = (D^{2},S^{1})$,  $i = 1,2,\ldots,m$.  At the core of subject of {\em toric topology\/}, which began as
a topological approach to toric geometry, are three families of spaces:
\begin{enumerate}
\item polyhedral products and  moment-angle manifolds 
\item Davis-Januszkiewicz spaces and
\item toric manifolds and real toric manifolds.
\end{enumerate}

\nd This paper is devoted to an analysis of the properties of these spaces  with respect to an operation
$$J\colon K \longrightarrow K(J)$$

\nd from abstract simplicial complexes with $m$ vertices to abstract simplicial complexes,   
which is determined by a sequence of positive integers  $J = (j_1, j_2,\ldots,j_m)$. Here, the simplicial complex
$K(J)$ has $d(J) = j_{1}+j_{2}+\cdots+j_{m}$ vertices. Several applications to toric topology
are discussed in detail.

By abuse of notation, the symbol
$Z\big(K(J);(\underline{X}, \underline{A})\big)$ is used to denote the polyhedral product
determined by the simplicial complex $K(J)$ and the family of pairs obtained from 
$(\underline{X}, \underline{A})$ by repeating each $(X_{i},A_{i})$, $j_{i}$ times in sequence. 
In Section \ref{sec:sw.macs}, another operation $\overline{J}$ is defined on families of CW-pairs yielding an identity of polyhedral products
$$Z\big(K(J); (\underline{X}, \underline{A})\big) = Z(K;\overline{J}(\underline{X}, \underline{A})\big).$$

\nd This result has consequences for the case
$$(\underline{X}, \underline{A})\;=\;(\underline{D}^{2J}, \underline{S}^{2J-1})\;=\;
\big\{({D}^{2j_i},{S}^{2j_i-1})\big\}^m_{i=1}$$

\nd corresponding to the polyhedral products which are called now  {\em generalized\/} moment-angle
complexes. The central result here is the following.

\begin{thma}
There is an action of $T^m$ on both $Z(K;(\underline{D}^{2J}, \underline{S}^{2J-1}))$ and
$Z\big(K(J); (D^2, S^1)\big)$, with respect to which they are equivariantly homeomorphic.
\end{thma}

\nd When combined with Theorem 1.8 of \cite{bbcg2},  this theorem yields an immediate corollary.

\begin{cora}
The spaces $Z(K;(D^2, S^1))$ and $Z\big(K(J); (D^2, S^1)\big)$ have isomorphic
ungraded cohomology rings.
\end{cora}
\nd Another application links the spaces $(\underline{D}^{2J}, \underline{S}^{2J-1})$ to the study of toric manifolds in a new way.

For the much studied case $(X_{i},A_{i}) = (D^{1},S^{0})$,  $i = 1,2,\ldots,m,\;$  it follows from the results
in Section \ref{sec:swcpp}  that every moment-angle complex $Z(K;(D^{2},S^{1}))$ 
can be realized as 
$Z\big(K(J);(D^{1},S^{0})\big)$ for  $J = (2,2,\ldots,2)$.  So, in a certain sense, the ``real'' 
moment-angle complex is the more basic object.
In the context of toric manifolds over a simple polytope \pn,  with simplicial complex
$K$ dual to the boundary  $\partial{P^{n}}$,  the spaces $Z(\partial{P^{n}};(D^{2},S^{1}))$
and $Z(\partial{P^{n}};(D^{1},S^{0}))$,  which in this case are differentiable manifolds,  
were introduced by Davis and Januszkiewicz in \cite{davis.jan}.
They used the latter to construct spaces known as {\em small covers\/}, the subject of considerable
current investigation.
\skp{0.15}

The fundamental  construction of Davis and Januszkiewicz 
\cite[Section $1.5$]{davis.jan}, realizes all toric manifolds $M^{2n}$ and in particular, all smooth projective toric 
varieties.
From this construction it follows that  $M^{2n}$ can be described
as the quotient of an $(m+n)$-dimensional moment-angle complex by the free action of a
certain real $(m-n)$-dimensional torus $T^{m-n} \subset T^{m}$.  This subtorus is specified usually
by a {\em characteristic map\/} $\lambda$.

Beginning with a toric manifold $M^{2n}$, its associated simple  polytope $P^n$ having $m$ facets
and characteristic map $\lambda$, an infinite family of new toric manifolds $M(J)$ is constructed, 
one for each sequence of positive integers $J = (j_1, j_2,\ldots,j_m)$. The manifolds $M(J)$ are 
determined by a new polytope $P(J)$ and a new characteristic 
map $\lambda(J)$. The 
outcome here (Theorem \ref{thm:condensed}), is that the integral cohomology ring 
of $M(J)$ is described completely in terms of the original map $\lambda$ and the original polytope \pn.
 The manifolds $M(J)$ are a substantial new concrete  class of toric manifolds which come 
 equipped with a complete
fan, (where appropriate), combinatorial and topological information. These spaces are a new, 
systematic infinite family of toric manifolds which have tractable as well as natural  properties.
Furthermore, these constructions arise from an operation defined on a category of finite simplicial 
complexes sending $K$ to $K(J)$ and so provide a natural construction.

\skp{0.155}
In Section \ref{sec:macs.tm}, the construction and properties of these manifolds $M(J)$ are analyzed in 
the context of  {\em generalized\/} moment-angle complexes. 
As above, to the  polytope $P^n$ is associated  the simplicial complex $K$ dual to  
$\partial{P^{n}}$ 
and a generalized moment-angle 
complex $Z(K;(\underline{D}^{2J}, \underline{S}^{2J-1}))$ and,  to  $P(J)$, which has boundary dual 
to $K(J)$,  is associated the moment-angle complex $Z\big(K(J); (D^2,S^1)\big)$. 
Theorem \ref{thm:key.result} connects the construction $K(J)$  to the study of toric manifolds. 
 In describing the new manifolds $M(J)$, diffeomorphisms, (Theorems  \ref{thm:key.result} and \ref{main.thm}), 
$$Z\big(K(J); (D^2,S^1)\big)\big/ T^{m-n} \longrightarrow M(J) 
\longleftarrow  Z(K;(\underline{D}^{2J}, \underline{S}^{2J-1}))\big/T^{m-n},$$

\nd mirror geometrically the reduction in 
combinatorial complexity from the pair $\big(\lambda(J),P(J)\big)$ to  the pair $(\lambda, P^{2n})$.
Significant in the theory of toric manifolds is the role played by the Davis-Januszkiewicz
spaces. These are homotopy equivalent to polyhedral products
of the form $Z(K; (\mathbb{C}P^\infty,\ast))$. Key in the theory which is developed here,
are related spaces, denoted by 
$Z\big(K;(\underline{\mathbb{C}P}^\infty, \underline{\mathbb{C}P}^{J-1})\big)$.
They substitute for the usual Davis-Januszkiewicz spaces 
$Z\big(K(J); (\mathbb{C}P^\infty,\ast)\big)$. Their properties  and   
relationship to the manifolds $M(J)$ are discussed in Theorems \ref{thm:cohom.dj.analog} and
\ref{thm:top.calculation}.  In contrast to the cohomology of the Davis-Januszkiewicz spaces,
these spaces have integral cohomology rings which are monomial ideal rings but the monomials are
not necessarily square-free.
\skp{0.1}
The construction of the toric manifolds $M(J)$ leads to the idea of ``nests'' of toric manifolds which is
discussed in the final section. 

\begin{rem}Unless indicated otherwise, all cohomology rings throughout are 
 taken with integral coefficients. 
\end{rem}
The intersections of certain quadrics are known to be diffeomorphic to   moment-angle manifolds, 
\cite{bm}, \cite{sam.santiago} and implicitly in \cite[Construction 3.1.8]{buchstaber.panov.1}.  
After a first draft of this article was written in 2008, the authors learned of the work of S.~Lopez de 
Medrano on the intersections of quadrics, \cite{ldm2}, \cite{ldm1}.  Results  in \cite{ldm1}  
depend on the consequences of a doubling of variables and a duplication
of  coefficients in the defining equations; this translates into an instance of the general 
construction given here. 

The simplicial wedge construction, in the 
context of toric varieties, has appeared in the work of G.~Ewald \cite{ewald}.  
The ideas presented are distinguished from those of \cite{ewald} by virtue of the following.
\begin{enumerate}
\item Not all the manifolds $M(J)$ discussed here are non-singular toric varieties.
\item The combinatorial constructions are analyzed in the context of polyhedral products,  culminating in Theorem \ref{thm:general.wedge}.
\item A theory is constructed for the manifolds $M(J)$ which realizes them as quotients of {\em generalized\/} moment-angle complexes and parallels the constructions developed for toric manifolds by
Davis-Januszkiewicz \cite{davis.jan} and Buschstaber-Panov \cite{buchstaber.panov.2}. 
\end{enumerate}

\nd Moreover, it is shown in the preprint \cite{bbcg4} that the manifolds $M(J)$ can be described directly by a generalization of Davis and Januszkiewicz's original construction. 

As is mentioned in the abstract, several applications of the constructions presented here have 
appeared both in the literature and in preprint form.
\skp{0.2}
\nd {\bf Acknowledgments.\/}
The authors are grateful to Carl Lee and Taras Panov for many interesting and helpful
conversations, in particular  for bringing to the authors' attention the fact that
the simplicial complex $K(v_i)$,  described by \eqref{const:wedge},  
is dual to the boundary of a simple
polytope if $K$ is.  Carl Lee provided the authors with an explicit proof of this
theorem before the discovery of the result in the work of Provan and Billera   \cite{pb}.
Originally, the authors had arrived at $K(v_i)$ from the point of view of 
the generalized moment-angle complexes which are discussed in Section \ref{sec:gmacs}. 
The authors thank Peter Landweber for his careful reading of manuscript and for his many
valuable suggestions. The authors are grateful to the referee for detecting various errors
and for suggestions which have improved the paper.

The first author was supported in part by the award of a Rider University Research Leave and
Research Fellowship and grant number 210386 from the Simons Foundation. The third 
author  thanks the Department of Mathematics at the University 
of Pennsylvania for partial support as well as for a fertile environment during the Fall of 2010. He
was partially supported also by {\sc DARPA\/} grant number 2006-06918-01. The fourth author 
received support from {\sc CONACYT}, Mexico. 

\section{A construction on simplicial complexes}\label{sec:new.complex}
Let $K$ be a simplicial complex of dimension $n-1$ on vertices 
$\{v_1,v_2,\ldots,v_m\}$ and let
$J = (j_1, j_2,\ldots,j_m)$ be a sequence of positive integers. 

\begin{defin}\label{defn:doubled.k}
A {\em minimal non-face\/} of a simplicial complex $K$ is a sequence of vertices of $K$ which is not
a simplex of $K$ but any proper subset is a simplex of $K$. 
\nd Let $K$ be as above.  Denote by  $K(J)$ the simplicial complex on 
$j_1 + j_2 + \cdots +j_m$ new vertices, labelled
$$v_{11}, v_{12},\ldots,v_{1j_1},\;v_{21}, v_{22},\ldots,v_{2j_2},\;\;
\ldots\;\; ,v_{m1}, v_{m2},\ldots,v_{mj_m},$$


\nd with the property that 
$$\big\{v_{i_{1}1},v_{i_{1}2},\ldots,v_{i_{1}j_{i_{1}}},
v_{i_{2}1},v_{i_{2}2},\ldots,v_{i_{2}j_{i_{2}}},\ldots,
v_{i_{k}1},v_{i_{k}2},\ldots,v_{i_{k}j_{i_{k}}}\big\}$$


\nd  is a  minimal non-face of $K(J)$ if and only if  $\{v_{i_1},v_{i_2},\ldots,v_{i_k}\}$ is a minimal 
non-face of $K$. Moreover, all minimal non-faces of $K(J)$ have this form. \end{defin}

An alternative construction of the simplicial complex $K(J)$ will reveal the fact that  $K(J)$ is
dual to the boundary of $P(J)$ if $K$ is dual to 
$\partial{P^n}$.  Recall that for $\sigma \in K$, the {\em link\/} of $\sigma$ in $ K$, is the set
$$\text{link}_{K}\sigma \;\becomes\; \{\tau \in K\colon \sigma \cup \tau \in K, \sigma\cap \tau = \emptyset\}.$$
The {\em join\/} of two simplicial complexes $K_1$, $K_2$ on disjoint vertex sets $\mathit{S}_1$ and
$\mathit{S}_2$ respectively is given by
$$K_1\ast K_2 \;\becomes \;\{\sigma \subset \mathit{S}_1 \cup \mathit{S}_2 
\colon \sigma = \sigma_1 \cup \sigma_2, \;\sigma_1 \in K_1, \;\sigma_2 \in K_2\}.$$

\begin{construction}\label{const:wedge}
 As above, let $K$ be the simplicial complex on vertices $\{v_1,v_2,\ldots,v_m\}$. Choose a fixed
vertex  $v_i$  in $K$  and define a new simplicial complex $K(v_i)$  with   $m+1$ vertices
labelled $v_{1}, \ldots, v_{i-1}, v_{i1}, v_{i2}, v_{i+1}, \ldots, v_{m}$, by
\begin{equation}\label{eqn:swc}
K(v_i) \;\becomes\; \{v_{i1},v_{i2}\} \ast \text{link}_{K}\{v_i\} \;\cup \;\big\{\{v_{i1}\},\{v_{i2}\}\big\} \ast 
(K\setminus \{v_i\}).
\end{equation}

\nd Next, the vertices of  $K(v_i)$, other
than $v_{i1}$ and $v_{i2}$,  are  re-labelled by setting $v_k = v_{k1}$ if $k \neq i$. So, the new vertex set
of  $K(v_i)$ becomes 
$$\mathit{S} = \{v_{11}, \ldots, v_{(i-1)1}, v_{i1}, v_{i2}, v_{(i+1)1}, \ldots, v_{m1}\}.$$ 
\end{construction} 

\begin{exm} 
The easiest example is that of $K = \big{\{v_{1}\},\{v_{2}\}}$ two disjoint points. Here, $K(v_{1})$ has
three vertices $ \{v_{11},v_{12},v_{21}\}$, $\text{link}_{K}\{v_1\} = \varnothing$ and 
$K\setminus \{v_1\} = v_{2}$. So \eqref{eqn:swc} becomes
\begin{equation}
\{v_{11},v_{12}\} \ast \varnothing  \;\cup \;\big\{\{v_{11}\},\{v_{12}\}\big\} \ast \{v_{2}\}
\;=\; \{v_{11},v_{12}\}  \;\cup\; \big(\{v_{11},v_{21}\} \;\cup\; \{v_{12},v_{21}\}  
\big)\end{equation}

\nd which is the boundary of a two-simplex.
\end{exm}

\nd \nd In  the paper \cite[page 578]{pb}, this construction is called the {\em simplicial wedge\/} of $K$ on $v$.
Notice that if $\{v_{i_1},v_{i_2},\ldots,v_{i_k}\}$ is a minimal  non-face  of $K$ with $i_j \neq i$ for
all $j$, then it remains a minimal  non-face  of $K(v_i)$. The simplex $\{v_{i1}, v_{i2}\}$  becomes 
part of a simplex 
$$\{v_{i_11},v_{i_21},\ldots,v_{i_k1},v_{i1}, v_{i2},v_{i_{(k+1)}1}, \ldots, v_{i_s1}\} \;
\in \; \{v_{i1}, v_{i2}\} \ast \text{link}_{K}\{v_i\} \;\subseteq \;K(v_i)$$

\nd if and only if 
$$\{v_{i_1},v_{i_2},\ldots,v_{i_{k}},v_i,v_{i_{k+1}}, \ldots, v_{i_s}\} \in K.$$

\nd Hence, according to Definition \ref{defn:doubled.k}, $K(v_i) = K(J)$, where 
$$J = (1,1, \ldots,1,2,1,\ldots,1),$$

\nd is the $m$-tuple with ``2'' appearing in the $i^{\text{th}}$ coordinate.   According to \cite[page 582]{pb},
$K(v_i)$ is  dual to the boundary of a simple polytope $P(v_i)$  of dimension $n+1$ with 
$m+1$ facets if $K$ is
dual to the boundary of a simple polytope $P^n$  of dimension $n$ with $m$ facets. Beginning with 
$J = (1,1,\ldots,1)$, Construction \ref{const:wedge} may be iterated to produce $K(J)$ for any 
$J  = (j_1, j_2,\ldots,j_m)$. The induction from $J  = (j_1, j_2,\ldots,j_m)$ to the new sequence
$J'  = (j_1, j_2,\ldots,j_{i-1},j_i+1,j_{i+1},\ldots,j_m)$, necessitates  a choice of vertex $v$
from among $\{v_{i1},v_{i2},\ldots,v_{ij_i}\}$ in order to form $K(J)(v)$, as in Construction
\ref{const:wedge}. The fundamental property, described in  Definition \ref{defn:doubled.k},
ensures that any choice of $v$, will  result in precisely the same minimal non-simplices
in $K(J)(v) = K(J')$.

The next theorem follows from these observations. 
Set $d(J) = j_1 + j_2 +\cdots + j_m$.
\begin{thm}\label{thm:iterated.wedge}
Let $J  = (j_1, j_2,\ldots,j_m)$ and suppose $K$  is dual to the boundary of a simple polytope 
\pn  having $m$ facets.
Then $K(J)$ is dual to the boundary of a simple polytope $P(J)$  of dimension $d(J)-m+n$ 
having $d(J)$ facets.  \hfill 
$\square$
\end{thm}

This section ends with a simple necessary criterion for a simplicial complex to be in the
image of the simplicial wedge construction. The condition follows immediately from the definition of 
the construction

\begin{rem}
If  a simplicial complex $K'$ exists satisfying, $K = K'(v)$ for some $\{v\} \in K'$, then $K$ must contain
vertices $v_1$ and $v_2$ satisfying:
\begin{enumerate}
\item the one-simplex $\{v_1,v_2\} \in K$ and
\item interchanging $v_1$ and $v_2$ is a simplicial automorphism of $K$.
\end{enumerate}
\end{rem}

\section{New toric manifolds made from a  given  one}
A toric manifold $M^{2n}$ is a manifold covered by local 
charts $\mathbb{C}^n$, each with the standard  action of a real $n$-dimensional torus
$T^n$, compatible in such a way that 
the quotient $M^{2n}\big/T^n$ has the structure of a  {\em simple\/} polytope \pn. Here, ``simple'' 
means that  $P^n$  has the property that at each vertex, exactly $n$ facets  intersect.
Under the $T^n$ action, each copy of $\mathbb{C}^n$ must project to an $\mathbb{R}^n_+$  
neighborhood of a vertex of $P^n$. The fundamental  construction of Davis and Januszkiewicz 
\cite[Section $1.5$]{davis.jan} is described briefly below. It realizes all toric manifolds and, in 
particular, all smooth projective toric varieties. Let
$${\mathcal F} = \{F_{1},F_{2},\ldots,F_{m}\}$$

\nd denote the set of facets of \pn. The fact that \pn $\;$ is 
simple implies that every codimension-$l$ face $F$ can be 
written uniquely as
$$F = F_{i_{1}} \cap F_{i_{2}} \cap \cdots \cap F_{i_{l}}$$

\nd where the $F_{i_{j}}$ are the facets containing $F$. Let
\begin{equation}\label{eqn:lambda}
\lambda : {\mathcal F} \longrightarrow \mathbb{Z}^n
\end{equation}

\nd be a function into an $n$-dimensional integer lattice satisfying the 
condition that whenever 
$F = F_{i_{1}} \cap F_{i_{2}} \cap \cdots \cap F_{i_{l}}$ then 
$\{\lambda(F_{i_{1}}),\lambda(F_{i_{2}}), \ldots ,\lambda(F_{i_{l}})\}$ span 
an $l$-dimensional submodule of $\mathbb{Z}^n$ which is a direct summand.
Such a map is called a {\em characteristic function\/} associated to \pn.
Next, regarding $\mathbb{R}^n$ as the Lie algebra of $T^n$,  the map
$\lambda$ is used to associate  to each codimension-$l$ face $F$ of \pn \/ a rank-$l$
subgroup $G_F \subset T^n$. Specifically, writing
$\lambda(F_{i_j}) = (\lambda_{1{i_j}},\lambda_{2{i_j}},\ldots,\lambda_{n{i_j}})$ gives
$$ G_F = \big\{\big(e^{2\pi{i}(\lambda_{1{i_1}}t_1 +\lambda_{1{i_2}}t_2 + \cdots + \lambda_{1{i_l}}t_l)},
\ldots, e^{2\pi{i}(\lambda_{n{i_1}}t_1 + \lambda_{n{i_2}}t_2 + \cdots +\lambda_{n{i_l}}t_l)}\big) \in T^n\big\}$$

\nd where $t_i \in \mathbb{R},\, i = 1,2,\ldots,l$. Finally, let $p \in$ \pn $\;$ and $F(p)$ be 
the unique face with $p$ in its relative interior. Define an equivalence
relation $\sim$ on $T^n$ $\times$ \pn $\;$ by $(g,p) \sim (h,q)$ if and only
if $p = q$ and $g^{-1}h \in G_{F(p)} \cong T^l$. Then
\begin{equation}\label{eqn:defn.tm}
M^{2n} \cong M^{2n}(\lambda) = T^n \times P^n\big/\!\sim
\end{equation} 

\nd is a smooth, closed, connected, $2n$-dimensional manifold with 
$T^n$ action induced by left translation \cite[page 423]{davis.jan}. A 
projection $\pi \colon M^{2n} \rightarrow P^n$ onto the polytope is induced from the projection
$T^n \times$ \pn $\rightarrow$ \pn. 
\begin{rem}
In the cases when $M^{2n}$ is a projective non-singular toric
variety, \pn \;and $\lambda$ encode the information in the defining fan.
\end{rem}
Suppose that $K$  is dual to to the boundary a simple polytope $P^n$  
having $m$ facets. Recall
that the duality here is in the sense that the facets of $P^n$ correspond to the vertices of 
$K$.  A set of vertices in $K$ is a simplex if and only if the corresponding facets in
$P^n$ all intersect.  At each vertex of a simple
polytope $P^n$, exactly $n$ facets intersect. 

A characteristic function 
$\lambda \colon \mathcal{F} \longrightarrow \mathbb{Z}^n$,   assigns an integer vector
to each facet of the simple polytope $P^n$. It can be considered as an $(n \times m)$-matrix, 
$\lambda \colon \mathbb{Z}^m \longrightarrow \mathbb{Z}^n$,
with integer entries and columns indexed by the facets of  $P^n$.  The condition following
(\ref{eqn:lambda}) may be interpreted as requiring all $n \times n$ minors of $\lambda$, 
corresponding to the vertices of $P^n$, to be $\pm1$. 
\skp{0.15}
Given $\lambda$ and $J  = (j_1, j_2,\ldots,j_m)$ , a new function
$$\lambda(J) \colon \mathbb{Z}^{d(J)} \longrightarrow \mathbb{Z}^{d(J)-m+n}$$

\nd can be constructed by taking $\lambda(J)$ to be the $\big(\big(d(J)-m+n\big) \times d(J)\big)$-matrix 
described in Figure \!1 below.
In the diagram, the columns of the matrix are indexed by the vertices 
of $K(J)$ and $I_k$  denotes a $k\times k$ identity sub-matrix.

\newpage
\setlength{\unitlength}{8.5mm}
\newcounter{qn}
\hoffset=0.0in
\setlength{\oddsidemargin}{12pt}
\skp{0.1}
\begin{picture}(10,50)

\put(0.2,48.3){$\bm{v_{12}}$}
\put(1,48.3){$\cdots$}
\put(2,48.3){$\bm{v_{1j_1}}$}

\put(0,44){\framebox(3,4)[c]{\huge{$I_{j_{1}-1}$}}}
\put(0,40){\framebox(3,4)[c]{\huge{$0$}}}
\put(0,36){\framebox(3,4)[c]
{\begin{picture}(4,4)
\put(2,1){$\centerdot$}
\put(2,2){$\centerdot$}
\put(2,3){$\centerdot$}
\end{picture}}}
\put(0,32){\framebox(3,4)[c]{\huge{$0$}}}
\put(0,28){\framebox(3,4)[c]{\huge{$0$}}}

\put(3.2,48.3){$\bm{v_{22}}$}
\put(4,48.3){$\cdots$}
\put(5,48.3){$\bm{v_{2j_2}}$}

\put(3,44){\framebox(3,4)[c]{\huge{$0$}}}
\put(3,40){\framebox(3,4)[c]{\huge{$I_{j_{2}-1}$}}}
\put(3,36){\framebox(3,4)[c]{\huge{$0$}}}
\put(3,32){\framebox(3,4)[c]
{\begin{picture}(4,4)
\put(2,1){$\centerdot$}
\put(2,2){$\centerdot$}
\put(2,3){$\centerdot$}
\end{picture}}}
\put(3,28){\framebox(3,4)[c]{\huge{$0$}}}

\put(7.15,48.3){$\cdots$}

\put(6,44){\framebox(3,4)[c]{$\centerdot\;\centerdot\;\centerdot$}}
\put(6,40){\framebox(3,4)[c]{\huge{$0$}}}
\put(6,36){\framebox(3,4)[c]
{\begin{picture}(4,4)
\put(2.5,1){$\centerdot$}
\put(1.85,2){$\centerdot$}
\put(1.2,3){$\centerdot$}
\end{picture}}}
\put(6,32){\framebox(3,4)[c]{\huge{$0$}}}
\put(6,28){\framebox(3,4)[c]{\huge{$0$}}}

\put(9.2,48.3){$\bm{v_{m2}}$}
\put(10.1,48.3){$\cdots$}
\put(10.8,48.3){$\bm{v_{mj_m}}$}

\put(9,44){\framebox(3,4)[c]{\huge{$0$}}}
\put(9,40){\framebox(3,4)[c]{\huge{$0$}}}
\put(9,36){\framebox(3,4)[c]{\huge{$0$}}}
\put(9,32){\framebox(3,4)[c]{\huge{$I_{j_{m}-1}$}}}
\put(9,28){\framebox(3,4)[c]{\huge{$0$}}}

\put(12.3,48.3){$\bm{v_{11}}$}
\put(13.2,48.3){$\bm{v_{21}}$}
\put(14.2,48.3){$\cdots$}
\put(15.1,48.3){$\bm{v_{m1}}$}
\put(12,44){\framebox(4,4)[l]{\begin{picture}(4,4)
\put(0.1,3.5){$-1$}
\put(0.1,3){$-1$}
\put(0.1,1.3){\begin{picture}(1,1)
\put(0.36,1){$\centerdot$}
\put(0.36,0.5){$\centerdot$}
\put(0.36,0){$\centerdot$}
\end{picture}}
\put(0.1,0.2){$-1$}
\end{picture}}}
\put(14,45.65){{\huge{$0$}}}

\put(12,40){\framebox(4,4)[l]{\begin{picture}(4,4)
\put(0.2,3.5){$0$}
\put(0.2,3){$0$}
\put(0.2,1.3){\begin{picture}(1,1)
\put(0.2,1){$\centerdot$}
\put(0.2,0.5){$\centerdot$}
\put(0.2,0){$\centerdot$}
\end{picture}}
\put(0.2,0.2){$0$}
\put(0.8,3.5){$-1$}
\put(0.8,3){$-1$}
\put(0.5,1.3){\begin{picture}(1,1)
\put(0.7,1){$\centerdot$}
\put(0.7,0.5){$\centerdot$}
\put(0.7,0){$\centerdot$}
\end{picture}}
\put(0.8,0.2){$-1$}
\end{picture}}}
\put(14,41.65){{\huge{$0$}}}

\put(12,36){\framebox(4,4)[c]
{\begin{picture}(4,4)
\put(2,1){$\centerdot$}
\put(2,2){$\centerdot$}
\put(2,3){$\centerdot$}
\end{picture}}}

\put(12,32){\framebox(4,4)[l]{\begin{picture}(4,4)
\put(1.9,1.7){{\huge{$0$}}}
\put(3.2,3.5){$-1$}
\put(3.2,3){$-1$}
\put(3,1.3){\begin{picture}(1,1)
\put(0.57,1){$\centerdot$}
\put(0.57,0.5){$\centerdot$}
\put(0.57,0){$\centerdot$}
\end{picture}}
\put(3.2,0.2){$-1$}
\end{picture}}}

\put(12,28){\framebox(4,4)[c]{\huge{$\lambda$}}}
\put(12.3,27.5){$\bm{1}$}
\put(13,27.5){$\bm{2}$}
\put(14,27.5){$\bm{\cdots}$}
\put(15.4,27.5){$\bm{m}$}

\put(16.2,31.5){$\bm{1}$}
\put(16.2,30.9){$\bm{2}$}
\put(15.7,29.25){\begin{picture}(1,1)
\put(0.57,1){$\centerdot$}
\put(0.57,0.5){$\centerdot$}
\put(0.57,0){$\centerdot$}
\end{picture}}
\put(16.2,28.15){$\bm{n}$}

\put(4.6,25.8){\large{{\bf Figure 1.}} The matrix $\lambda(J)$}
\end{picture}

\newpage
\hoffset=0in
\setlength{\oddsidemargin}{0pt}
\begin{remark}
It was brought to the authors' attention by Jongbaek Song that the two 
$\big((n+2) \times (m+2)\big)$-matrices, $\lambda(3,1,\ldots,1)$
and $\overline{\lambda}(2,1,\ldots,1)$, where $\overline{\lambda} = \lambda(2,1,\ldots,1)$, differ by an
element of $SL(n+2,\mathbb{Z})$. This allows for inductive arguments in this context of toric manifolds.
\end{remark}

The next theorem constructs an 
infinite family of toric manifolds  ``derived" from the information in $\lambda$, $P^n$ and  
$J  = (j_1, j_2,\ldots,j_m)$.

\begin{thm}\label{thm:char.map}
If $\lambda$ is a characteristic map for a $2n$-dimensional toric manifold $M$, then
$\lambda(J)$ is the characteristic map for a toric manifold $M(J)$ of dimension
$2d(J) - 2m + 2n$.
\end{thm}
\begin{proof}
Theorem \ref{thm:iterated.wedge} ensures that $P(J)$ is a simple polytope of dimension
$d(J) -m +n$.  It remains to show that for each vertex of  $P(J)$,  the 
corresponding $\big(d(J) -m +n\big) \times \big(d(J) -m +n\big)$ minor of  $\lambda(J)$  is
equal to $\pm1$.
\nd The proof is by induction.  Consider first the case 
$J = (1,1,\ldots,1,\overset{i}{2},1,\ldots,1)$. Corresponding to the 
re-indexing of the vertices of $K(J)$, the facets of $P(J)$ are indexed as follows

$$\mathcal{F}(J) = \{F_{11}, F_{21}, \ldots, F_{(i-1)1},F_{i1},F_{i2},F_{(i+1)1},\ldots,F_{m1}\}$$

\

\nd The matrix $\lambda(J)$ now has the form

\

\begin{tabbing}
mmmmmmmmmmmmmi$F_{i2}$\=e$F_{11}$\=mmmmm$F_{i1}$\=mmmmn$F_{m1}$\=mm\kill
\>$F_{i2}$\>$F_{11}$\>$F_{i1}$\>$F_{m1}$ 
\end{tabbing}
\skp{-0.6}
{\Large \begin{equation*}\begin{pmatrix}
1             &\!\!0                     &\cdots  &\!\!-1                 & 0          &0 \\
0             & \lambda_{11} &\cdots & \lambda_{1i} & \cdots &\lambda_{1m}\\
0             &\vdots              &\vdots & \vdots           & \cdots &\vdots\\
0             & \lambda_{n1} &\cdots & \lambda_{ni} & \cdots &\lambda_{nm}\\
                 
\end{pmatrix}\end{equation*}}
                 
\

\nd where $\lambda = (\lambda_{ij})$ is the original matrix. (Recall that vertices $v_{ik}$ of
$K(J)$ correspond to facets $F_{ik}$ of the polytope $P(J)$.)
The minors corresponding to the new $(n+1)$-fold intersections of facets,
are  of two types.
\begin{enumerate}
\item Those which include columns indexed by both $F_{i1}$ and  $F_{i2}$.
\item Those which include columns indexed by either $F_{i1}$ or $F_{i2}$ but not both.
\end{enumerate}

\

\nd This observation follows from the fact that the simplicial wedge construction ensures that
each new maximal simplex of $K(J) = K(v_i)$ must contain either $v_{i1}$ or $v_{i2}$.
According to the discussion following Construction \ref{const:wedge}, the first type arise
from intersections 

\begin{equation}\label{eqn:intersection}
F_{i_1} \cap F_{i_2} \cap \cdots \cap F_{i_k} \cap F_{i} \cap F_{i_{k+1}} \cap \cdots \cap F_{i_n}
\end{equation}

\

\nd of $n$ facets in $P^n$. In $P(J)$, they give $(n+1)$-fold intersections

$$F_{i_11} \cap F_{i_21} \cap \cdots \cap F_{i_k1} \cap F_{i1} \cap F_{i2} \cap F_{i_{(k+1)1}} \cap 
\cdots \cap F_{i_n1}.$$

\

\nd The corresponding $(n+1) \times (n+1)$ minors in the matrix $\lambda(J)$ above are

\

\samepage{\begin{tabbing}
mmmmmmmmii$F_{i2}$\=e$F_{11}$\=mmmmmn$F_{i1}$\=mm$F_{i1}$\=mm$F_{i1}$
\=mmmmmn$F_{m1}$\=mm\kill
\>$F_{i2}$\>$F_{i_{1}1}$\>$F_{i_{k}1}$\>$F_{i1}$\>$F_{i_{k+1}1}$\>$F_{i_{n}1}$ 
\end{tabbing}
\skp{-0.2}
{\Large \begin{equation*}\begin{vmatrix}
1             &\!\!\!0                     &\cdots   &\! \!\!0       &\!\!-1       &\!\!\!\!0          & 0          &\!\!\!0 \\
0  &\lambda_{1i_1} &\cdots    &\lambda_{1i_k}&\lambda_{1i} &\lambda_{1i_{k+1}}  & \cdots &\lambda_{1i_n}\\
\vdots            &\vdots    &\vdots  &\vdots    & \vdots &\vdots   & \cdots &\vdots\\
0  & \lambda_{ni_1} &\cdots  &\lambda_{ni_k} & \lambda_{ni} &\lambda_{1i_{k+1}} & \cdots &\lambda_{ni_n}\\
 \end{vmatrix}\end{equation*}}}

\

\nd Expanding by the first row gives $\pm1$ by (\ref{eqn:intersection}).  Now  $(n+1)$-fold intersections 
of facets of the second type, which contain $F_{i2}$ but not $F_{i1}$ (or vice versa) arise from intersections
in P which do not involve the facet $F_i$. That is, they are of the form

\begin{equation}\label{eqn:facets}
F_{i_1} \cap F_{i_2} \cap \cdots \cap F_{i_k} \cap F_{i_{k+1}} \cap \cdots \cap F_{i_n}, \quad 
i_j \neq i
\end{equation}

\nd In $P(J)$ the intersection is 
$$F_{12} \cap F_{i_11} \cap F_{i_21} \cap \cdots \cap F_{i_k1} \cap F_{i_{(k+1)1}} \cap 
\cdots \cap F_{i_n1}, \quad i_j \neq i .$$

\nd The $(n+1) \times (n+1)$ minor in $\lambda(J)$ will have the form 

\

{\begin{tabbing}
mmmmmmmmmmmmmmmmi$F_{i2}$\=e$F_{11}$\=mmmmmi$F_{i1}$
\=mmmmmn$F_{m1}$\=mm\kill
\>$F_{i2}$\>$F_{i_{1}1}$\>$F_{i_{n}1}$ 
\end{tabbing}
\skp{-0.2}
{\Large \begin{equation*}\begin{vmatrix}
1             &\!\!\!0                     &\cdots          &\!\!\!0 \\
0             &\lambda_{1i_1}     &\cdots          &\lambda_{1i_n}\\
\vdots      &\vdots                   &          &\vdots\\
0             & \lambda_{ni_1}    &\cdots          &\lambda_{ni_n}\\
 \end{vmatrix}\end{equation*}}}
 
 \
 
 \nd and so have the value $\pm1$. Finally,  $(n+1)$-fold intersections of facets of the second type, 
 which contain $F_{i1}$ but not $F_{i2}$ again arise from intersections of the type \ref{eqn:facets}.
In $P(J)$ the intersection is 
$$F_{i_11} \cap F_{i_21} \cap \cdots \cap F_{i_k1} \cap F_{i1} \cap F_{i_{(k+1)1}} \cap 
\cdots \cap F_{i_n1}.$$

\

\nd and the corresponding $(n+1) \times (n+1)$ minor in $\lambda(J)$ 
has the form 

\

{\begin{tabbing}
mmmmmmmmnm$F_{i2}$\=mmmmmm$F_{11}$\=mm$F_{i1}$\=nm$F_{i1}$\=mmmmmm$F_{i1}$
\=mmmmmmmm$F_{m1}$\=mm\kill
\>$F_{i_{1}1}$\>$F_{i_{k}1}$\>$F_{i1}$\>$F_{i_{k+1}1}$\>$F_{i_{n}1}$ 
\end{tabbing}
\skp{-0.2}
{\Large \begin{equation*}\begin{vmatrix}
\!\!\!0                 &\cdots    &\! \!\!0              &\!\!-1               &\!\!\!\!0                      & \cdots &\!\!\!0 \\
\lambda_{1i_1} &\cdots    &\lambda_{1i_k}&\lambda_{1i} &\lambda_{1i_{k+1}}  & \cdots &\lambda_{1i_n}\\
\vdots                &             &\vdots              & \vdots           &\vdots                       &           &\vdots\\
 \lambda_{ni_1} &\cdots  &\lambda_{ni_k} & \lambda_{ni} &\lambda_{1i_{k+1}} & \cdots &\lambda_{ni_n}\\
 \end{vmatrix}\end{equation*}}}
 
 \
 
 \nd Expansion by the first row gives the $n\times n$ minor in $\lambda$

{\Large \begin{equation*}\begin{vmatrix}
\lambda_{1i_1} &\cdots    &\lambda_{1i_k} &\lambda_{1i_{k+1}}  & \cdots &\lambda_{1i_n}\\
\vdots                &             &\vdots              & \vdots                       &            &\vdots  \\
\lambda_{ni_1} &\cdots  &\lambda_{ni_k}  &\lambda_{1i_{k+1}} & \cdots &\lambda_{ni_n}
\end{vmatrix}\end{equation*}}

\nd which has the value $\pm 1$ because it corresponds to (\ref{eqn:facets}).  
 The inductive step passes from the $m$-tuple 
$J  = (j_1, j_2,\ldots,j_m)$ to $J'  = (j_1, j_2,\ldots, j_{k-1}, j_k +1,j_{k+1} \ldots, j_m)$ and follows the
same argument, replacing the characteristic map $\lambda$ in the discussion above with $\lambda(J)$,
This completes the proof. 
\end{proof} 

\begin{remark}
In a recent preprint \cite{bbcg4},   the authors  show that the manifolds $M(J)$ can be obtained alternatively by a reinterpretation and generalization of the Davis-Januszkiewicz construction \eqref{eqn:defn.tm}.
 \end{remark}\
\section{The cohomology of the toric manifolds $M(J)$}

The rows of the matrix $\lambda(J)$ determine an ideal $L_{M(J)}$ generated by
linear relations among the generators of the Stanley-Reisner ring of $K(J)$. These are given by:

\begin{align}\label{eqn:linear.relations}
v_{it} - v_{i1} &= 0 , \quad t = 2, \ldots, j_i, \quad i = 1, \ldots,m\notag\\
\\
\lambda_{i1}v_{11} + \lambda_{i2}v_{21} + \ldots + \lambda_{im}v_{m1} &= 0, \quad i = 1, \ldots,n
\notag \end{align}

\

\nd Notice that the second set of relations are those corresponding to the linear ideal determined by
the matrix $\lambda$. The next result is the Davis-Januszkiewicz, (Danilov-Jurkewicz) theorem,
\cite[Theorem 4.14]{davis.jan}, for the toric manifold $M(J)$.

\begin{thm}\label{thm:standard.description}
The cohomology ring $H^*(M(J); \mathbb{Z})$ is isomorphic to
$$\mathbb{Z}[v_{11}, v_{12},\ldots,v_{1j_1}, v_{21}, v_{22},\ldots,v_{2j_2},
\;\ldots\;,v_{m1}, v_{m2},\ldots,v_{mj_m}]\Big/\big(I_{K(J)} + L_{M(J)}\big)$$

\

\nd where $I_{K(J)}$ denotes the Stanley-Reisner ideal for the simplicial complex
$K(J)$. \hfill $\square$ 
\end{thm}

\

\nd Applying the linear relations (\ref{eqn:linear.relations}) and rewriting $v_{i1}$ as 
$v_i$, allows a significant simplification of this description.

\begin{thm}\label{thm:condensed}
The cohomology ring $H^*(M(J); \mathbb{Z})$ is isomorphic to
$$\mathbb{Z}[v_1,v_2, \ldots, v_m]\big/(I^{J}_{K} + L_M)$$

\nd where each $v_i$ has degree two, $L_M$ is the ideal in the Stanley-Reisner ring 
of $K$ generated by the rows of the matrix $\lambda$ and $I^{J}_K$ is the ideal of relations 
generated by all monomials of the form
\begin{equation}\label{eqn:new.monomials}
v_{i_1}^{j_{i_1}}v_{i_2}^{j_{i_2}}\cdots v_{i_k}^{j_{i_k}}
\end{equation}

\nd corresponding to the minimal non-simplex $\{v_{i_1},v_{i_2},\ldots,v_{i_k}\}$ of $K$. 
\end{thm}

\begin{proof}
Linear relations (\ref{eqn:linear.relations}) and the relabeling  of $v_{i1}$ as 
$v_i$, convert the monomials generating $I_{K(J)}$ into those of (\ref{eqn:new.monomials})
and, the relations $L_{M(J)}$ into the relations $L_M$.
\end{proof}


 \section{A  class  of examples}\label{sec:cpcp}
In this section is constructed a family of toric manifolds beginning with the example $M^{4}$ for which the simple
polytope $P^{2}$ is the two-dimensional square having four facets and the characteristic map $\lambda$ is given by the $2\times 4$-matrix
$$\lambda = \begin{pmatrix}
                      1&-1&1&0\\
                      2&-1&0&1 \end{pmatrix}.$$

\nd The cohomology of $M^{4}$ is computed from the Davis-Januszkiewicz theorem \cite[Theorem 4.14]{davis.jan}.
It is the quotient of the polynomial ring $\mathbb{Z}[v_{1},v_{2},v_{3},v_{4}]$ on four two-dimensional
generators  by a linear ideal $L_{M^{4}}$ and a monomial ideal $I_{K}$, where $K$ is the simplicial complex 
 dual to  $\partial{P^{2}}$. The ideal $L_{M^{4}}$ is generated by
$v_{1} - v_{2}+v_{3}$ and $2v_{1} - v_{2} +  v_{4}$  and the ideal $I_{K}$ is generated by $v_{1}v_{3}$ and $v_{2}v_{4}$. These relations give the computation

$$H^{*}(M^{4}) \cong \mathbb{Z}[v_{1},v_{3}]\big/\langle v_{1}v_{3}, \;v_{1}^{2} = v_{3}^{2}\rangle.$$

\

\nd This is the cohomology ring of the connected sum $\mathbb{C}P^{2}\#\mathbb{C}P^{2}$, so the cohomological
rigidity results of \cite{choi.park.suh} imply that $M^{4}$ is homeomorphic to $\mathbb{C}P^{2}\#\mathbb{C}P^{2}$.

Next,  take $J = (1,p,1,q)$  to get $M(J)$ ($= M^{4}(J)$).  Theorem \ref{thm:condensed} gives $H^{*}(M(J))$ as
the quotient of the polynomial ring $\mathbb{Z}[v_{1},v_{2},v_{3},v_{4}]$ by the linear ideal $L_{M^{4}}$ and a monomial ideal $I^{J}_{K}$. As before, the ideal $L_{M^{4}}$ is generated by
$v_{1} - v_{2}+v_{3}$ and $2v_{1} - v_{2} +  v_{4}$  and the ideal $I^{J}_{K}$ is generated by 
$v_{1}v_{3}$ and $v_{2}^{p}v_{4}^{q}$. 
These relations are simplified by using the fact that $v_{1}v_{3}=0$ to get

$$(v_{3}-v_{1})^{p} =v_{3}^{p}+ (-1)^{p}v_{1}^{p} \quad \text{and}\quad 
(v_{3}+v_{1})^{q} = v_{3}^{q}+v_{1}^{q}.$$

\

\nd Writing $0=v_{2}^{p}v_{4}^{q} = (v_{3}-v_{1})^{p}(v_{3}+v_{1})^{q}$ gives

$$0=(v_{3}^{p}+(-1)^{p}v_{1}^{p})(v_{3}^{q}+v_{1}^{q}) = v_{3}^{p+q}+ (-1)^{p}v_{1}^{p+q}.$$

\

\nd Thus, $v_{3}^{p+q} = (-1)^{p+1}v_{1}^{p+q}$ giving

$$H^{*}\big(M(J)\big) \cong 
\mathbb{Z}[v_{1},v_{3}]\big/\langle v_{1}v_{3}, v_{3}^{p+q} = (-1)^{p+1}v_{1}^{p+q}\rangle.$$

\

\nd This is the cohomology ring of the connected sum $\mathbb{C}P^{p+q}\#(-1)^{p}\mathbb{C}P^{p+q}$
where here, the $(-1)^{p}$ indicates an orientation change with $p$. Again the cohomological
rigidity results of \cite{choi.park.suh} imply that $M^{J}$ is homeomorphic to 
$\mathbb{C}P^{p+q}\#(-1)^{p}\mathbb{C}P^{p+q}$.

\begin{rem} The situation becomes more complex for general $J = (j_{1},j_{2},j_{3},j_{4})$, with all
$j_{i}>1$. The ``heights'' of $v_{1}$ and $v_{3}$ are harder to determine. It would interesting
to know if for this particular manifold $M^{4}$, non-trivial examples of sequences $J$ and $J'$ exist so 
that $M(J)$ and $M(J')$ have isomorphic cohomology rings.

J.~Song and S.~Choi have produced examples of non-diffeomorphic toric manifolds $M$ and $N$
and a sequence $J$ so that $M(J)$ and $N(J)$ {\em are\/} diffeomorphic.
\end{rem}
\section{Polyhedral products} \label{sec:gmacs}
Let $K$ be a simplicial complex with $m$ vertices and let $(\underline{X}, \underline{A})$ 
denote a family of CW pairs
$$(X_1,A_1), (X_2,A_2), \ldots ,(X_m,A_m).$$

\

\nd When all the pairs $(X_i,A_i)$ are the same pair $(X,A)$, the family
$(\underline{X}, \underline{A})$ is written simply as $(X,A)$.
A  polyhedral product, is a topological space 
$$Z(K; (\underline{X}, \underline{A})) \; \subseteq \; \prod\limits_{i=1}^{m}{X_i}\;,$$

\nd defined as a colimit by a diagram 
$D: K \to CW_{\ast}$. At each $\sigma \in K$, 
it is given by

\begin{equation}\label{eqn:d.sigma}
D(\sigma) =\prod^m_{i=1}W_i,\quad {\rm where}\quad
W_i=\left\{\begin{array}{lcl}
X_i  &{\rm if} & i\in \sigma\\
A_i &{\rm if} & i\in [m]-\sigma.
\end{array}\right.
\end{equation}

\

\nd Here, the colimit is a union given by 
$$Z(K; (\underline{X}, \underline{A})) = \bigcup_{\sigma \in K}D(\sigma).$$

\nd Detailed background information about polyhedral products
may be  found in  \cite{buchstaber.panov.1}, \cite{buchstaber.panov.2},  
\cite{bbcgpnas}, \cite{bbcg} (and  in the unpublished notes of N.~Strickland). 

\

The family of $CW$-pairs $(\underline{X}, \underline{A})$ to be investigated here,
is specified by the sequence of positive integers  $J=(j_1,\ldots,j_m)$ and is given by
\begin{equation}\label{eqn:big.disks}
(\underline{X}, \underline{A})\; =\; (\underline{D}^{2J}, \underline{S}^{2J-1}) \; = \;
\big\{({D}^{2j_i},{S}^{2j_i-1})\big\}^m_{i=1}
\end{equation}

\nd (Presently, it will become necessary to consider the discs $D^{2j_i}$ as embedded 
naturally in $\mathbb{C}^{j_i}$.) If $J = (1,1,\ldots,1)$, the space
 $Z(K;(\underline{D}^{2J}, \underline{S}^{2J-1}))$ is an ordinary moment-angle complex
 and is written $Z(K;(D^2,S^1))$.
 
 \

For fixed $K$, the spaces $Z(K;(\underline{D}^{2J}, \underline{S}^{2J-1}))$
all have the property of being {\em stably wedge equivalent\/}.

\begin{defin}\label{defin:wedge.equiv.}
Two spaces $X$ and $Y$ are said to be stably wedge equivalent if there are
stable homotopy equivalences
$$X\sim X_1 \vee X_2 \vee \ldots \vee X_t \quad \text{and} \quad
Y\sim Y_1 \vee Y_2 \vee \ldots \vee Y_t$$

\nd and  $X_i$ is stably homotopy equivalent to $Y_i$ \;for all $i = 1,2,\ldots,t$.
\end{defin}

\

The next proposition follows directly from the  stable splitting theorems for generalized
moment-angle complexes of \cite{bbcgpnas}, \cite[Corollary 2.24]{bbcg} and 
\cite[Theorem 1.4]{bbcg2} which describes the cohomology ring structure in terms of the
stable splitting.

\begin{prop}\label{prop:wedge.equivalent}
For fixed $K$ and all  $J=(j_1,\ldots,j_m)$, the spaces $Z(K;(\underline{D}^{2J}, \underline{S}^{2J-1}))$
are all stably wedge equivalent and moreover, they have isomorphic ungraded
cohomology rings. \hfill $\square$
\end{prop} 

\

\section{The simplicial wedge construction and polyhedral products}\label{sec:swcpp}
\label{sec:sw.macs}

Recall that a {\em product\/} of CW pairs is defined by
\begin{equation}\label{eqn:pairs.prod}
(X,A) \times (Y,B) \; \becomes   \; \big(X\times Y, (X\times B) \cup (A\times Y)\big).
\end{equation}

\nd The $k$-fold iteration $(X,A) \times \cdots \times (X,A)$ is denoted by $(X,A)^k$.

\

Let $K$ be a simplicial complex on $m$ vertices $\{v_1,v_2,\ldots,v_m\}$ and let 
$(\underline{X}, \underline{A})$ denote the family of CW pairs 
$$(X_1,A_1), (X_2,A_2), \ldots ,(X_m,A_m).$$

\

\nd In the light of Definition \ref{defn:doubled.k}, it becomes necessary at this point to introduce a 
notational convention to avoid expressions becoming too unwieldy.

 \begin{convention}\label{conv:conv}  Let $J = (j_1, j_2,\ldots,j_m)$ be  sequence of positive 
integers, $K(J)$ as in Definition \ref{defn:doubled.k} and the family of pairs $(\underline{X}, \underline{A})$
as above.  Denote by $Z\big(K(J); (\underline{X}, \underline{A})\big)$ the polyhedral product
determined by the simplicial complex $K(J)$ and the family of pairs obtained from 
$(\underline{X}, \underline{A})$ by repeating each $(X_{i},A_{i})$, $j_{i}$ times in sequence.
\end{convention}

\

\nd Fix $i \in \{1,2,\ldots,m\}$ and define a family of CW pairs $(\underline{Y}, \underline{B})$ 
by

$$(Y_k,B_k) =\left\{\begin{array}{lcl}
(X_k,A_k)  &{\rm if} &k \neq i\\
(X_i,A_i)^2 &{\rm if} &k=i.
\end{array}\right.$$

\

\begin{thm}\label{thm:wedge.mac}
The polyhedral product spaces 
$Z(K;(\underline{Y}, \underline{B}))$ and $Z\big(K(v_i); (\underline{X}, \underline{A})\big)$
are equal subspaces of $W_1 \times \cdots \times W_m$, where

$$W_k =\left\{\begin{array}{lcl}
X_k &{\rm if} &k \neq i\\
X_k \times X_k,  &{\rm if} &k=i.
\end{array}\right.$$
\end{thm}

\begin{proof}
As in Construction \ref{const:wedge}, the vertices of $K(v_i)$ are
$$\mathit{S} = \{v_{11}, \ldots, v_{(i-1)1}, v_{i1}, v_{i2}, v_{(i+1)1}, \ldots, v_{m1}\}.$$

\nd Let $\sigma = \{v_{t_1}, v_{t_2}, \dots,v_{t_k}\}$ be a maximal simplex in $K$. 
If $v_i  = v_{t_s} \in \sigma$, then $D(\sigma)$  is equal to
$D(\sigma') \subset Z\big(K(v_i); (\underline{X}, \underline{A})\big)$ where

$$\sigma' \;=\; \{v_{t_11}, v_{t_21}, \dots,v_{(t_{s}-1)1},v_{t_s1},v_{t_s2},v_{(t_{s}+1)1},
\ldots,v_{t_k1}\}.$$

\

\nd If $v_i  \notin \sigma$, then $D(\sigma)$ is identified, (by the identity map), with 
$D(\sigma'_1) \cup D(\sigma'_2) \;\subset\; Z\big(K(v_i); (\underline{X}, \underline{A})\big)$, where
$$\sigma'_1 \;=\; \{v_{t_11}, v_{t_21}, \dots,v_{t_k1},v_{i1}\}.$$
$$\sigma'_2 \;=\; \{v_{t_11}, v_{t_21}, \dots,v_{t_k1},v_{i2}\}.$$

\

\nd (Here, the vertices may not be in their correct order.)  The fact 
that  the wedge construction ensures that  the maximal simplices  $\sigma'_1$ and $\sigma'_2$ 
exist in $K(v_i)$ has been used here.  This procedure exhausts all maximal simplices 
in  $K(v_i)$ and completes the equivalence.
\end{proof}

The next iteration of this procedure is slightly more involved but serves to describe
all iterations.  Choose $j \in \{1,2,\ldots,m\}$.
If $j \neq i$, the new family  $(\underline{Y}, \underline{B})$ is defined by

$$(Y_k,B_k) =\left\{\begin{array}{lcl}
(X_k,A_k)  &{\rm if} &k \neq i,j\\
(X_k,A_k)^2 &{\rm if} &k=i,j.
\end{array}\right.$$

\

\nd In this case, $K(v_i)$ is replaced with $\big(K(v_i)\big)(v_{j1})$ and the procedure is exactly 
as described above. The result is that
$Z(K;(\underline{Y}, \underline{B}))$ and $Z\big((K(v_i))(v_{j1}); (\underline{X}, \underline{A})\big)$
are equivalent subspaces of $W_1 \times \cdots \times W_m$, where

$$W_k =\left\{\begin{array}{lcl}
X_k &{\rm if} &k \neq i,j\\
X_k \times X_k,  &{\rm if} &k=i,j.
\end{array}\right.$$

\

\nd In the case $j=i$, the new family  $(\underline{Y}, \underline{B})$ will have

$$(Y_k,B_k) =\left\{\begin{array}{lcl}
(X_k,A_k)  &{\rm if} &k \neq i\\
(X_i,A_i)^3 &{\rm if} &k=i.
\end{array}\right.$$

\

\nd In this case, $K(v_i)$ is replaced with either $\big(K(v_i)\big)(v_{i1})$ or $\big(K(v_i)\big)(v_{i2})$.
The symmetry of the wedge construction ensures that these two simplicial complexes are
isomorphic as simplicial complexes. The result now is that
$Z(K;(\underline{Y}, \underline{B}))$ and $Z\big((K(v_i))(v_{i1}); (\underline{X}, \underline{A})\big)$
are equivalent subspaces of $W_1 \times \cdots \times W_m$, where

$$W_k =\left\{\begin{array}{lcl}
X_k &{\rm if} &k \neq i\\
\prod_{s=1}^3{X_k}  &{\rm if} &k=i.
\end{array}\right.$$

\

\nd The general iteration procedure is now straightforward; the result is recorded in the
next theorem.

\

\samepage{\begin{thm}\label{thm:general.wedge}
Let $K$ be a simplicial complex with $m$ vertices and let $(\underline{X}, \underline{A})$ 
denote a family of CW pairs
$$\{(X_1,A_1), (X_2,A_2), \ldots ,(X_m,A_m)\}.$$
For  $J = (j_1, j_2,\ldots,j_m)$, a sequence of positive integers, define a new family of pairs
$$ \overline{J}(\underline{X}, \underline{A})   =  (\underline{Y}, \underline{B})) \quad \text{by} \quad 
(Y_k,B_k) \;=\; (X_k,A_k)^{j_k}, \quad k = 1,2,\ldots,m. $$

\nd Then, the polyhedral product spaces 
$ Z(K;\overline{J}(\underline{X}, \underline{A}))$ and $Z\big(K(J); (\underline{X}, \underline{A})\big)$
are equal subspaces of
$X_1^{j_1} \times X_2^{j_2} \times \cdots \times X_m^{j_m}$. \hfill $\square$
\end{thm}}

\

This result is applied to the family where $(X_i,A_i) = (D^2, S^1)$ for all $i = 1,2,\ldots,m$. 
In this case,

$$(Y_k,B_k) \;=\; (X_k, A_k)^{j_k} \;=\; \big((D^2)^{j_k}, \partial\big((D^2)^{j_k}\big)\big),$$

\ 

\nd where $\partial{\big((D^2)^{j_i}}\big)$ denotes the boundary of a $j_i$-fold product
of two-discs. For this particular case, the family $(\underline{Y}, \underline{B})$ is denoted by

$$(\underline{B}^{2J}, \underline{\partial{B}}^{2J})\;\becomes \; \big\{\big((D^2)^{j_i}, 
\partial{\big((D^2)^{j_i}}\big)\big\}_{i=1}^{m}.$$

\

\nd Theorem \ref{thm:general.wedge} implies  now  the next corollary.

\begin{cor}\label{cor:d2.case}
The generalized moment-angle complex $Z(K;(\underline{B}^{2J}, \underline{\partial{B}}^{2J}))$ and 
the moment-angle complex $Z\big(K(J); (D^2, S^1)\big)$
are equal subspaces of
\begin{equation*}
(D^2)^{j_1} \times (D^2)^{j_2} \times \cdots \times (D^2)^{j_m} = (D^2)^{d(J)}. 
\end{equation*}
\end{cor}
\skp{0.1}
\begin{rem}
Notice that by considering $(D^2)^{d(J)} \subset \mathbb{C}^{d(J)}$, both moment-angle 
complexes inherit an action of the real torus $T^{d(J)}$ with respect to which they are equivariantly
equivalent.
\end{rem}

\nd  An entirely similar argument shows that taking $(X_i,A_i) = (D^1, S^0)$ for all $i = 1,2,\ldots,m\;$
and $J = (2,2,\ldots,2)$, yields the result that $Z\big(K;(D^{1}\times D^{1}, \partial(D^{1}\times D^{1})\big)$
and $Z(K(J);(D^{1},S^{0}))$ are  equivalent subspaces of 
$(D^{1})^{2}\times (D^{1})^{2}\times \cdots \times (D^{1})^{2}$. It follows by the arguments below that 
$Z(K;(D^{2},S^{1}))$ and $Z(K(J);(D^{1},S^{0}))$ are diffeomorphic.

\

Recall now the family of pairs $(\underline{D}^{2J}, \underline{S}^{2J-1})$ described in 
(\ref{eqn:big.disks}). Observe that for the corresponding generalized moment-angle complex, there is a 
natural embedding,

$$Z(K;(\underline{D}^{2J}, \underline{S}^{2J-1})) \;\subseteq\; D^{2j_1} \times D^{2j_2} \times
\cdots \times D^{2j_m}$$
\

\nd The next goal is to verify that the generalized moment-angle complexes
$Z(K;(\underline{B}^{2J}, \underline{\partial{B}}^{2J}))$ and 
$Z(K;(\underline{D}^{2J}, \underline{S}^{2J-1}))$ are equivariantly homeomorphic with respect
to various torus actions.

\

Let $(D^2)^{j_i} \subset \mathbb{C}^{j_i}$ be embedded in the usual way and choose a
standard  homeomorphisms of pairs
$h_i\colon    \big((D^2)^{j_i}, \partial{\big((D^2)^{j_i}}\big)  \longrightarrow (D^{2j_i},S^{2j_{i}-1})$.  
Define an action of the circle  $T^1$ on 
$D^{2j_i}$ by 

$$ t\cdot h_i(z_1,z_2, \ldots, z_{j_i}) \; = \; h_i(tz_1,tz_2, \ldots, tz_{j_i}).$$

\

\nd Next, denote the $j_i$-tuple $(z_1,z_2, \ldots, z_{j_i})$ by the symbol $\underline{z}_{j_i}$
and define an action of $T^m$ on $(D^2)^{d(J)}$ by

\begin{equation}\label{eqn:tm1.action}
(t_1,t_2,\ldots,t_m)\cdot (\underline{z}_{j_1},\underline{z}_{j_2},\dots,\underline{z}_{j_m})
\;=\; (t_1\underline{z}_{j_1},t_2\underline{z}_{j_2},\dots,t_m\underline{z}_{j_m})
\end{equation}

\

\nd where $t_i\underline{z}_{j_i}$ has the usual meaning $(t_iz_1,t_iz_2, \ldots, t_iz_{j_i})$.
Notice that this action of $T^m$ is a restriction of the natural action of the torus $T^{d(J)}$
on $(D^2)^{d(J)}$. An action of $T^m$ is induced on 
$D^{2j_1} \times D^{2j_2} \times \cdots \times D^{2j_m}$ by

\begin{equation}\label{eqn:tm.action}
(t_1,t_2,\ldots,t_m)\cdot \big(h_1(\underline{z}_{j_1}), h_2(\underline{z}_{j_2}),\ldots,
h_m(\underline{z}_{j_m})\big) =
\big(h_1(t_1\underline{z}_{j_1}), h_2(t_2\underline{z}_{j_2}),\ldots,
h_m(t_m\underline{z}_{j_m})\big).
\end{equation}

\

\nd The  homeomorphisms $h_i$ give rise to a homeomorphism, 
equivariant with respect to the 
$T^m$-action above, 
\begin{equation}\label{eqn;disks.diffeo}
H \colon (D^2)^{d(J)} \longrightarrow D^{2j_1} \times D^{2j_2} \times \cdots \times D^{2j_m}
\end{equation}

\nd by $H(\underline{z}_{j_1},\underline{z}_{j_2},\dots,\underline{z}_{j_m}) =
\big(h_1(\underline{z}_{j_1}), h_2(\underline{z}_{j_2}),\ldots,
h_m(\underline{z}_{j_m})\big)$.
The homeomorphism $H$ extends to a homeomorphism of generalized moment-angle complexes.

\begin{lem}\label{lem:diffeo.of.disc.gmacs}
The  homeomorphism $H$ extends to a homeomorphism of 
generalized moment-angle complexes
$Z(K;(\underline{B}^{2J}, \underline{\partial{B}}^{2J}))$ and 
$Z(K;(\underline{D}^{2J}, \underline{S}^{2J-1}))$,  which is equivariant with respect to the 
$T^m$-action defined by (\ref{eqn:tm.action}). 
\end{lem}
\begin{proof}
The map $H$ induces a homeomorphisms at the level of the appropriate spaces $D(\sigma)$
defined in \eqref{eqn:d.sigma}, compatible with the maps defining both colimits.
\end{proof}
\nd Combining this observation with Corollary \ref{cor:d2.case} and the remark following it,
yields a key result.

\begin{thm}\label{thm:key.result}
There is an action of $T^m$ on both $Z(K;(\underline{D}^{2J}, \underline{S}^{2J-1}))$ and
$Z\big(K(J); (D^2, S^1)\big)$, with respect to which they are equivariantly  homeomorphic.
\hfill $\square$
\end{thm}

When combined with Theorem 1.8 of \cite{bbcg2},  this theorem yields an immediate corollary.

\begin{cor}\label{cor:ungraded.iso}
The spaces $Z(K;(D^2, S^1))$ and $Z\big(K(J); (D^2, S^1)\big)$ have isomorphic
ungraded cohomology rings. \hfill $\square$
\end{cor}

These results yield an observation about the action of the Steenrod algebra.

\begin{cor}\label{cor:steenrod}
There is an isomorphism of ungraded $\mathbb{Z}/2$-modules
$$H^*\big(Z(K;(D^2, S^1));\mathbb{Z}/2 \big) \longrightarrow 
H^*\big(Z\big(K(J); (D^2, S^1)\big);\mathbb{Z}/2 \big)$$

\nd which commutes with the action of the Steenrod algebra. 
\end{cor}

\begin{proof}
The Steenrod operations are stable operations and hence the splitting theorem 
\cite[Theorem 2.21]{bbcg} implies that there is an isomorphism of ungraded $\mathbb{Z}/2$-modules
$$H^*\big(Z(K;(D^2, S^1));\mathbb{Z}/2 \big)
\longrightarrow H^*\big(Z(K;(\underline{D}^{2J}, \underline{S}^{2J-1}));\mathbb{Z}/2\big)$$

\nd which commutes with the action of the Steenrod algebra. The result follows from Theorem
\ref{thm:key.result}. \end{proof} 

\begin{rem}
Corollary \ref{cor:steenrod} holds equally well for $\mathbb{Z}/p$ with $p$ an odd prime.
\end{rem}

\

\section{A generalization to  topological joins}\label{sec:generalization}
As usual, let $K$ be a simplicial complex on $m$ vertices and $J = (j_1, j_2,\ldots,j_m)$, 
a sequence of positive integers. Consider the family of pairs

$$(\underline{CX}, \underline{X}) \;=\; \big\{(CX_i,X_i)\big\}_{i=1}^m$$

\

\nd where, each space $X_i$ is a CW complex and $CX_i$ denotes the cone on $X_i$. 
Applying Theorem \ref{thm:general.wedge} to this family of pairs yields an equivalence of 
polyhedral products
\skp{0.01}
\begin{equation}\label{eqn:first.equiv}
Z(K;\overline{J}(\underline{CX}, \underline{X})) \longrightarrow Z\big(K(J);(\underline{CX}, \underline{X})\big)  
\end{equation} 
\skp{0.3}
\nd  where, as before, $\overline{J}(\underline{CX}, \underline{X}) = (\underline{Y}, \underline{B})$\; 
with\; $(Y_k,B_k) \;=\; (CX_k,X_k)^{j_k}, \; k = 1,2,\ldots,m$.

\

\nd The equivalence is as subspaces of 
$(CX_1)^{j_1} \times (CX_2)^{j_2} \times \cdots \times (CX_m)^{j_m}$. The next
lemma, which is well known, is a key ingredient.

\

\begin{lem}\label{lem:join}
For  any finite CW complex  $X$,  there is a homeomorphism of pairs
$$\big(C(X\ast X), X\ast X\big) \overset{f}{\longrightarrow} (CX, X)^2,$$

\nd where $\ast$ denotes the topological join.
\end{lem} 

\begin{proof}
Represent a point in $C(X\ast X)$ by $\big[s,[x_1,t,x_2]\big]$. Define the homeomorphism
$f$ by

$$f\big(\big[s,[x_1,t,x_2]\big]\big) \;=\; 
\big([2s\cdot \min\{t,1/2\},x_1],[2s\cdot \min\{1-t,1/2\},x_2]\big) \in CX \times CX,$$

\

\nd where the cone point is at $s=0$. At $s=1$, $f$ is the usual homeomorphism 

$$X\ast X \longrightarrow (CX \times X) \cup (X \times CX).$$

\

\nd The map $f$ is a continuous bijection between compact Hausdorff spaces and
hence is a homeomorphism. \end{proof}

\

\nd  Next, define a family of CW pairs by

$$\big(\underline{C(\divideontimes_{\!J}X}),\; \underline{\divideontimes_{\!J}X}\big)
\;\becomes\; \big\{\big(C(\underset{j_i}{\underbrace{X_i\ast X_i\ast\cdots\ast X_i}}),
\underset{j_i}{\underbrace{X_i\ast X_i\ast\cdots\ast X_i}}\big)\big\}_{i=1}^m.$$ 

\

\nd The map $f$ of Lemma \ref{lem:join} iterates easily to produce a homeomorphism
of pairs

$$\big(C(\underset{j_k}{\underbrace{X_k\ast X_k\ast\cdots\ast X_k}}),
\underset{j_k}{\underbrace{X_k\ast X_k\ast\cdots\ast X_k}}\big)
\overset{f_{j_k}}{\longrightarrow} (CX_k,X_k)^{j_k},$$

\

\nd which extends to a map of families of pairs,

$$\big(\underline{C(\divideontimes_{\!J}X}),\; \underline{\divideontimes_{\!J}X}\big)
\overset{f_{J}}{\longrightarrow} (\underline{Y}, \underline{B})).$$

\

The results above combine now to give the next theorem.

\

\begin{thm}
There is a homeomorphism of polyhedral products,
$$Z\Big(K; \big(\underline{C(\divideontimes_{\!J}X}),\; \underline{\divideontimes_{\!J}X}\big)\Big) 
\longrightarrow Z\big(K(J); (\underline{CX}, \underline{X})\big). $$              
\end{thm}

\begin{proof}
The homeomorphism is given by composing the homeomorphism of 
polyhedral products induced by $f_J$ with the homeomorphism 
of \eqref{eqn:first.equiv}.\end{proof}

\

\section{Toric manifolds and generalized moment-angle complexes}
\label{sec:macs.tm}
The information in a toric manifold $M^{2n}$ can be recorded concisely as a 
triple $(P^n, \lambda, M^{2n})$.   Let $K$ be the simplicial complex dual to the boundary of the
polytope $P^n$.  The moment-angle manifold  $Z(K;(D^2,S^1))$
is a subcomplex of the product of two-discs

$$Z(K;(D^2,S^1)) \subseteq (D^2)^m \subset \mathbb{C}^m.$$

\

\nd As such, it has a natural action of the real $m$-torus $T^m$. If $\lambda$ satisfies the
condition following (\ref{eqn:lambda}), then the kernel of $\lambda$, a subgroup 
$T^{m-n} \subset T^m$, acts freely on $Z(K;(D^2,S^1))$ and the quotient is homeomorphic to
$M^{2n}$, \cite[page 434]{davis.jan}, and \cite[page 86]{buchstaber.panov.2}. The action
of $T^n$ on $M^{2n}$, which yields $P^n$ as orbit space, is that given by the action of the
quotient  $T^m/T^{m-n}$ on $Z(K;(D^2,S^1))\big/T^{m-n}$.

\

The freeness of the action of $T^{m-n}$ on $Z(K;(D^2,S^1))$ implies a homotopy equivalence
of Borel spaces

\begin{equation}\label{eqn:borel.homeo}
ET^m\times_{T^m}Z(K;(D^2,S^1)) \longrightarrow ET^n\times_{T^n}M^{2n}.
\end{equation}

\

\nd Moreover, for any simplicial complex $K$, there is an equivalence 
(\cite{buchstaber.panov.2}), 

\begin{equation}\label{eqn:ds}
ET^m \times_{T^m}  Z(K;(D^2,S^1))  \; \cong \; 
Z(K; (\mathbb{C}P^\infty, \ast)),\end{equation}

\

\nd where the right hand side is a polyhedral product which is a subcomplex of
the product  space $(\mathbb{C}P^\infty)^m$. These spaces are called, (loosely),  the 
Davis-Januszkiewicz spaces associated to the simplicial complex $K$ and are denoted by the
symbol $\mathcal D\mathcal J(K)$. Also, the cohomology ring
$H^*(\mathcal D\mathcal J(K);\mathbb{Z})$
is the Stanley-Reisner ring of the simplicial complex $K$. The Serre spectral sequence of the
fibration

\begin{equation}\label{eqn:basic.fibration}
M^{2n} \longrightarrow  ET^n \times_{T^n} M^{2n}  
\stackrel{p}{\longrightarrow}\ BT^n \end{equation}

\

\nd yields an entirely topological computation of the ring $H^*(M;\mathbb{Z})$. Known as the
Davis-Januszkiewicz theorem, it generalizes the Danilov-Jurkewicz theorem for
projective non-singular toric varieties, \cite{davis.jan}. Applied to the manifolds
$M(J)$, it gives Theorem \ref{thm:standard.description}.

\

Given  $(P^n, \lambda, M^{2n})$, let $\lambda(J)$ be the matrix defined by Figure 1.
Choosing a standard basis and following \cite{buchstaber.panov.2}, the kernel of $\lambda$,  
as a sub-torus of $T^m$, is specified by an $m \times (m-n)$-matrix $S = (s_{ij})$. Explicitly, 
it is given by

\begin{equation*}\label{eqn.little.s}
\text{ker}\;{\lambda}\; = \;\big\{\big(e^{2\pi{i}(s_{11}\phi_1 + \cdots + s_{1(m-n)}\phi_{m-n})},\ldots,
e^{2\pi{i}(s_{m1}\phi_1 + \cdots + s_{m(m-n)}\phi_{m-n})}\big) \in T^m\big\}
\end{equation*}

\

\nd where $\phi_{i} \in \mathbb{R}, i = 1,2,\ldots,m-n$. The form of the matrix in
Figure 1 reveals the kernel of $\lambda(J)$ to be specified by the $d(J) \times 
(m-n)$-matrix $S(J) = (s^{J}_{lk})$, where


\begin{equation}\label{eqn:kernel}
s^{J}_{l k}= \begin{cases}
s_{1k}, \quad \text{if} \quad 1\leq l \leq j_{1}\\
s_{2k}, \quad \text{if} \quad  j_{1}+1\leq l \leq j_{1}+j_{2}\\
s_{3k}, \quad \text{if} \quad   j_{1}+j_{2}+1\leq l \leq j_{1}+j_{2}+j_{3}\\
\cdots\\
s_{mk}, \quad \text{if} \quad   j_{1}+j_{2}+\cdots+j_{m-1}+1\leq l \leq j_{1}+j_{2}+j_{3}+\cdots +j_{m} = d(J).
\end{cases}
\end{equation}

\

\nd Notice that this description makes explicit the isomorphism: $ \text{ker}\!\;{\lambda(J)}
\cong \text{ker}\;{\lambda} \cong  T^{m-n}$.


\

\nd The action of $\text{ker}\;\!\lambda(J)$ on both $Z(K;(\underline{D}^{2J}, 
\underline{S}^{2J-1}))$ and
$Z\big(K(J); (D^2, S^1))\big)$ is via the inclusions

$$\text{ker}\;\!\lambda(J)  \;\subset\; T^m \;\subset \;T^{d(J)},$$

\

\nd where the first inclusion is determined by the matrix $S$ and \eqref{eqn:kernel}; the
second is determined by \eqref{eqn:tm.action}. Now, $\text{ker}\;\!\lambda(J)$ acts freely on
$Z\big(K(J); (D^2, S^1)\big)$ by Theorem \ref{thm:char.map} and hence on 
$Z(K;(\underline{D}^{2J}, \underline{S}^{2J-1}))$ by Theorem \ref{thm:key.result} which implies
also the result following.

\begin{lem}\label{lem:mj.from.mac} 
There is a   homeomorphism of orbit spaces
$$Z\big(K(J); (D^2, S^1))\big)\big/\text{{\rm ker}}\;\lambda(J) \longrightarrow
Z(K;(\underline{D}^{2J}, \underline{S}^{2J-1}))\big/\text{{\rm ker}}\;\!\lambda(J).$$
\end{lem}

\begin{rem}
The space on the left acquires a differentiable structure from the general topological theory of
toric manifolds, \cite{davis.jan, buchstaber.panov.new.book}. The homeomorphism of 
Lemma \ref{lem:mj.from.mac} imposes then a differentiable structure on the space on the right.
\end{rem}
\skp{0.08}
\nd Recall now that Theorem \ref{thm:char.map} and the discussion at the beginning of this 
section imply that $M(J)$ is diffeomorphic to 
$Z\big(K(J); (D^2, S^1))\big)\big/\text{ker}\;\lambda(J)$. The main theorem follows
from Lemma \ref{lem:mj.from.mac}. 

\begin{thm}\label{main.thm}
There is a diffeomorphism
$$M(J) \longrightarrow Z(K;(\underline{D}^{2J}, \underline{S}^{2J-1}))\big/\text{{\rm ker}}\!\;\lambda(J).$$
\end{thm}

\

\section{The analogue of the Davis-Januszkiewicz space}\label{sec:dj.analogue}
\nd The fibration \eqref{eqn:basic.fibration} is used in the standard theory to exhibit 
$H^*(M^{2n};\mathbb{Z})$ as a quotient of the cohomology of the Davis-Januszkiewicz space.
In its various guises, this space is

$$\mathcal D\mathcal J(K)  \becomes   ET^n\times_{T^n} M^{2n} \;\simeq\;
ET^m \times_{T^m} Z(K;(D^2,S^1)) \;\simeq\; Z(K;(\mathbb{C}P^\infty, \ast)).$$

\

\nd The recognition of  $M(J)$ as the quotient  $Z(K(J);(D^2,S^1))\big/\text{ker}\!\;\lambda(J)$,
yields the cohomology calculation in Theorem \ref{thm:standard.description}. To get the more 
concise calculation afforded by Theorem \ref{thm:condensed}, the description of $M(J)$ given by
Theorem \ref{main.thm} is used instead. So, an appropriate analogue of the space 
$\mathcal D\mathcal J(K)$ is needed. 

\

Define a family of CW pairs by
$$(\underline{\mathbb{C}P}^\infty, \underline{\mathbb{C}P}^{J-1}) \becomes
\big\{(\mathbb{C}P^\infty, \mathbb{C}P^{j_1-1}), \ldots, 
(\mathbb{C}P^\infty, \mathbb{C}P^{j_m-1})\big\}$$

\

\nd and consider the polyhedral product

$$Z(K;(\underline{\mathbb{C}P}^\infty, \underline{\mathbb{C}P}^{J-1})) \;\subseteq \;
(\mathbb{C}P^\infty)^m \;= \;BT^m.$$

\begin{thm}\label{thm:dj.analogue}
There is a homotopy equivalence
$$ET^m\times_{T^m} Z(K;(\underline{D}^{2J}, \underline{S}^{2J-1})) 
\stackrel{\alpha}{\longrightarrow}
Z(K;(\underline{\mathbb{C}P}^\infty, \underline{\mathbb{C}P}^{J-1})),$$

\nd making the following diagram commute:

\begin{equation*}\
\begin{CD}
ET^m\times_{T^m} Z(K;(\underline{D}^{2J}, \underline{S}^{2J-1})) @>{\alpha}>>
Z(K;(\underline{\mathbb{C}P}^\infty, \underline{\mathbb{C}P}^{J-1})).\\
@VV{p_1}V                                          @V{i}VV \\
BT^m
@>{=}>>
BT^m.
\end{CD}
\end{equation*}
\end{thm}

\

\begin{proof}
For the pair $(\underline{X}, \underline{A}) = (\underline{D}^{2J}, \underline{S}^{2J-1})$, consider
$D(\sigma)$ as in \eqref{eqn:d.sigma}. The action of $T^m$ leaves $D(\sigma)$ invariant, so

$$ET^m\times_{T^m} Z(K;(\underline{D}^{2J}, \underline{S}^{2J-1}))
\;= \;ET^m\times_{T^m}  \big(\cup_{\sigma \in K} D(\sigma)\big)
\;=\; \bigcup_{\sigma \in K} ET^m\times_{T^m}  D(\sigma).$$

\

\nd The torus $T^m$ acts diagonally, so it suffices to observe that, as a pair,

$$\big(ET^1\times_{T^1}D^{2j_k}, \;ET^1\times_{T^1}S^{2j_k-1}\big) \;\simeq\; 
(\mathbb{C}P^\infty,  \mathbb{C}P^{j_k-1}).$$

\

\nd  This follows by an argument based on the fact that the left hand side is a 
disc--sphere bundle pair over $\mathbb{C}P^\infty$.

\

$D^{2j_k}$ is contractible and $T^1$ acts freely on $S^{2j_k-1}$.
So, the Borel construction converts $D(\sigma)$ for the family of pairs 
$(\underline{D}^{2J}, \underline{S}^{2J-1})$ into  $D(\sigma)$ for the family of pairs 
$(\underline{\mathbb{C}P}^\infty, \underline{\mathbb{C}P}^{J-1})$.
Moreover, the map $\alpha$ is constructed as a factorization of $p_1$, so the diagram
does commute.\end{proof} 

\begin{remark}
It is necessary to record certain standard facts about cell decompositions and their implications
for the polyhedral product complexes 
$Z(K;(\underline{\mathbb{C}P}^\infty, \underline{\mathbb{C}P}^{J-1}))$. 
The classical example of  C.~Dowker \cite{dowker}, shows that some care is needed.

\

Let $J = (j_1,j_2, \ldots,j_m)$ be as above and  fix  $N > j_i-1, \; i = 1,2, \ldots, m$. 
For $\sigma \in K$, consider the space

\begin{equation}\label{eqn:bounded.d.sigma}
D^N(\sigma) =\prod^m_{i=1}W_i,\quad {\rm where}\quad
W_i=\left\{\begin{array}{lcl}
\mathbb{C}P^N &{\rm if} & i\in \sigma\\
\mathbb{C}P^{j_i-1} &{\rm if} & i\in [m]-\sigma.
\end{array}\right.
\end{equation}

\

\nd The compact spaces $\mathbb{C}P^N$ and $\mathbb{C}P^{j_i-1}, \; i = 1,2,\ldots,m\;$  
are each assumed to be given the CW decomposition with one cell in each even dimension 
up to the top. This induces a cell decomposition of
the product $D^N(\sigma)$,  with each cell homeomorphic to a product of cells of even dimension,
each in one of the spaces  $W_i$. The compactness implies that the
product topology and compactly generated topology agree. 

\

\nd Consider now the spaces
 $D(\sigma)$, \eqref{eqn:d.sigma}, in 
$Z(K;(\underline{\mathbb{C}P}^\infty, \underline{\mathbb{C}P}^{J-1}))$:

\begin{equation}\label{eqn:d.sigma.jwdj}
D(\sigma) =\prod^m_{i=1}W_i,\quad {\rm where}\quad
W_i=\left\{\begin{array}{lcl}
\mathbb{C}P^\infty &{\rm if} & i\in \sigma\\
\mathbb{C}P^{j_i-1} &{\rm if} & i\in [m]-\sigma.
\end{array}\right.
\end{equation}

\

\nd The spaces  $D(\sigma)$  are each a colimit, over increasing $N$, of the spaces $D^N(\sigma)$. 
The colimit is via compatible inclusions and so each space $D(\sigma)$ inherits a CW structure with
cells in even dimension. Finally,
$$Z(K;(\underline{\mathbb{C}P}^\infty, \underline{\mathbb{C}P}^{J-1})) \;=\; \bigcup_{\sigma \in K}D(\sigma)$$

\nd is a finite colimit of spaces  which have compatible cell structures on intersections, so inherits
a cell structure with cells concentrated in even dimension.
\end{remark}

From these considerations follows the next lemma.

\begin{lem}\label{lem:even.degree}
The inclusion of the subcomplex 
$Z(K;(\underline{\mathbb{C}P}^\infty, \underline{\mathbb{C}P}^{J-1})) \subseteq BT^m\;$ induces an
surjective map
$$H^*(BT^m;\mathbb{Z}) \longrightarrow
H^*\big(Z(K;(\underline{\mathbb{C}P}^\infty, \underline{\mathbb{C}P}^{J-1}));\mathbb{Z}\big).$$
\end{lem}

Let $I_K^J$ be the ideal in $\mathbb{Z}[v_1,v_2,\ldots,v_m]$  described in Theorem
\ref{thm:condensed}. It is generated by all monomials 
$v_{i_1}^{j_{i_1}}v_{i_2}^{j_{i_2}}\cdots v_{i_k}^{j_{i_k}}$
corresponding to minimal non-simplices $\{v_{i_1},v_{i_2},\ldots,v_{i_k}\}$ of $K$.

{\begin{defin}\label{def:j.weighted.sr}
Let $K$ be a simplicial complex on $m$ vertices and $J = (j_1, j_2,\ldots,j_m)$, 
a sequence of positive integers. The ring $\mathbb{Z}[v_1,v_2,\ldots,v_m]\big/I_K^J$
is called the $J$-weighted Stanley-Reisner ring of $K$ and is denoted by the symbol
$SR^J(K)$. The space $Z(K;(\underline{\mathbb{C}P}^\infty, \underline{\mathbb{C}P}^{J-1}))$
is called the $J$-weighted Davis-Januszkiewicz space and is denoted by 
$\mathcal D\mathcal J^J(K)$.
\end{defin}}

\nd \begin{thm}\label{thm:cohom.dj.analog}
There is an isomorphism of rings 
$$H^*(Z(K;(\underline{\mathbb{C}P}^\infty, \underline{\mathbb{C}P}^{J-1})); \mathbb{Z})
\longrightarrow \mathbb{Z}[v_1,v_2,\ldots,v_m]\big/I_K^J.$$
\end{thm}

\begin{proof}
Since $Z(K;(\underline{\mathbb{C}P}^\infty, \underline{\mathbb{C}P}^{J-1}))$ is a cellular subcomplex
of $(\mathbb{C}P^\infty)^{m}$, having cells in even degree only, the kernel of the surjective map
$$H^*(\mathbb{C}P^\infty)^{m};\mathbb{Z}) \longrightarrow
H^*\big(Z(K;(\underline{\mathbb{C}P}^\infty, \underline{\mathbb{C}P}^{J-1}));\mathbb{Z}\big)$$

\nd is determined by an argument entirely analogous to that of 
\cite[Proposition $4.3.1$]{buchstaber.panov.new.book}.
\end{proof}

\begin{remark}
This result has been extended by the authors in \cite{bbcg3} to realize geometrically a large class 
of monomial ideal rings using simplicial complexes. The construction itself generalizes to realize all
monomial ideal rings.
\end{remark}

\newpage
\section{Generalized moment-angle complexes and the cohomology of $M(J)$}\label{sec:macs.and.m}

The computation of $H^*(M^{2n};\mathbb{Z})$  in \cite{davis.jan} is generalized to recover Theorem 
\ref{thm:condensed}
directly from the results of the previous section. Traditionally, (\cite{davis.jan}, \cite{buchstaber.panov.2}),
the canonical  diagram of fibrations 
\begin{equation*}\label{cd:traditional.diagram}
\begin{CD}
\text{ker}\;\! \lambda(J)
@>{}>>
E(\text{ker}\;\! \lambda(J))  @>{}>>     B(\text{ker}\;\! \lambda(J))\\
@VV{}V           @VV{}V                               @VV{}V  \\
Z(K(J);(D^2,S^1)) @>{}>> ET^{d(J)}\times_{T^{d(J)}}  Z(K(J);(D^2,S^1))  @>{p_1'}>>   BT^{d(J)}\\
@VV{r'}V           @VV{q'\; \simeq}V                               @VV{B(\lambda(J))}V  \\
M(J) @>{i'}>> ET^{d(J)-m+n}\times_{T^{d(J)-m+n}}  M(J))  @>{p_2'}>>   BT^{d(J)-m+n}
\end{CD}
\end{equation*}

\nd is used, in conjunction with the Serre spectral sequence of the fibration in the bottom row, to obtain
the standard description of $H^*(M^{2n};\mathbb{Z})$ given by Theorem \ref{thm:standard.description}.

\

\nd The more condensed calculation in Theorem \ref{thm:condensed} is obtained by considering instead
the homotopy commutative diagram of fibrations

\begin{equation}\label{cd:new.diagram}
\begin{CD}
\text{ker}\;\! \lambda(J)
@>{}>>
E(\text{ker}\lambda(J))  @>{}>>     B(\text{ker}\;\! \lambda(J))\\
@VV{}V           @VV{}V                               @VV{}V  \\
Z(K;(\underline{D}^{2J}, \underline{S}^{2J-1}))  @>{}>> 
ET^{m}\times_{T^{m}}  Z(K;(\underline{D}^{2J}, \underline{S}^{2J-1}))  @>{p_1}>>   BT^m\\
@VV{r}V           @VV{q\; \simeq}V                               @VV{B\lambda}V  \\
M(J) @>{i}>> ET^{n}\times_{T^{n}}  M(J))  @>{p_2}>>   BT^{n}.
\end{CD}
\end{equation}

\

\nd In the diagram, $r$ is the map given by Theorem \ref{main.thm} and the insertion of
the map $B\lambda$ is possible by \eqref{eqn:kernel} and the comment following it. 
The equivalence $q$ is a consequence of the splitting,
$\;T^m \cong \text{ker}\;\!\lambda(J) \times T^n\;$  as topological groups and the fact that
$\text{ker}\;\!\lambda(J)$ acts freely on $Z(K;(\underline{D}^{2J}, \underline{S}^{2J-1}))$.

\

\nd Theorem \ref{thm:dj.analogue} allows the replacement, up to homotopy, of the fibration
along the bottom row of the diagram with
\begin{equation}\label{eqn:new.fib}
M(J) \longrightarrow Z(K;(\underline{\mathbb{C}P}^\infty, \underline{\mathbb{C}P}^{J-1})) 
\stackrel{p}{\longrightarrow} BT^n.
\end{equation} 

\nd These observations allow for an alternative approach to the calculation of Theorem
\ref{thm:condensed}.
\begin{thm}\label{thm:top.calculation}
There is an isomorphism of rings
$$H^*(M(J);\mathbb{Z}) \longrightarrow 
H^*\big(Z(K;(\underline{\mathbb{C}P}^\infty, \underline{\mathbb{C}P}^{J-1}));\mathbb{Z}\big)\big/L,$$ 

\nd where $L$ is the two-sided ideal generated by the image of the map $p^*$.
\end{thm}
\begin{proof}
In the Serre spectral sequence associated with \eqref{eqn:new.fib}, all groups are concentrated
in even degree. This is true for $H^*(M(J);\mathbb{Z})$ because $M(J)$ is a toric
manifold by Theorem \ref{thm:char.map}  and, for 
$H^*\big(Z(K;(\underline{\mathbb{C}P}^\infty, \underline{\mathbb{C}P}^{J-1}));\mathbb{Z}\big)$,
by Theorem \ref{thm:cohom.dj.analog}. The spectral sequence collapses. The $E_2$-term is 
given by
$$H^*(M(J)) \otimes H^*(BT^n).$$

\nd It follows that $H^*(M(J))$ is the quotient of 
$H^*\big(Z(K;(\underline{\mathbb{C}P}^\infty, \underline{\mathbb{C}P}^{J-1}))\big)$ by the two-sided
ideal generated by the image of $p^*$. \end{proof}

It remains to identify the ideal $L$ in Theorem \ref{thm:top.calculation}. With reference to the
right-hand bottom square in diagram \eqref{cd:new.diagram}, the map $p^*$ is the composition 
$(\alpha^{-1})^* \circ ( p_1^* \circ s^*)$,
where the map $\alpha$ is the equivalence of Theorem \ref{thm:dj.analogue}. So, the image of
$p^*$ is  the image of the composition
\begin{equation*}
\begin{CD}
H^*(BT^n)  @>{s^*}>>  H^*(BT^m)  @>{}>> SR^J(K)\\
@VV{\cong}V           @VV{\cong}V      @VV{\cong}V \\
\mathbb{Z}[u_1,u_2,\ldots,u_n] @>{}>> \mathbb{Z}[v_1,v_2,\ldots,v_m] @>{}>> \mathbb{Z}[v_1,v_2,\ldots,v_m]\big/I_K^J.
\end{CD}
\end{equation*}

\nd This is specified by \eqref{eqn:kernel} and generates  the ideal $L_M$ of 
Theorem \ref{thm:condensed},
determined by the rows of the original matrix $\lambda$. So, $L = L_M$ and 
Theorem \ref{thm:top.calculation} becomes 
\begin{equation}\label{eqn:cohom.mj}
H^*(M(J); \mathbb{Z}) \;\cong\; \mathbb{Z}[v_1,v_2, \ldots, v_m]\big/(I^{J}_{K} + L_M),
\end{equation}

\nd recovering Theorem \ref{thm:condensed} completely from the topological point
of view of generalized moment-angle complexes.
\section{Nests}\label{sec:nests}
Let  $J = (j_1,j_2\ldots,j_m)$ and $L = (l_1,l_2,\ldots,l_m)$ be sequences of positive integers.  The
direct product ordering  on such sequences is given by   $J < L$  if $j_i \leq l_i$ for all $i$ and, 
for at least one $k \in \{1,2,\ldots,m\}$,
$j_k  < l_k$.  If $J<L$ the inclusions $D^{2j_i} \subseteq D^{2l_i}$ induce an equivariant embedding
$$Z(K;(\underline{D}^{2J}, \underline{S}^{2J-1}))\; 
\subset\; Z(K;(\underline{D}^{2L}, \underline{S}^{2L-1}))$$

\nd and, consequently, an embedding $\zeta\colon M(J) 
\longrightarrow  M(L)$. The next proposition follows easily.
\begin{prop}\label{prop:normal bundles}
For $J < L$, the normal bundle of the embedding $\zeta$ is
\begin{equation}\label{eqn:normal.bdle}
\bigoplus_{i = 1}^{m}\;(l_i - j_i)\alpha_i,
\end{equation}

\nd where $\alpha_i$ is the line bundle over $M(J)$ with first Chern class $c_i(\alpha_i) = v_i$,
the class from the cohomology description \eqref{eqn:cohom.mj}. Moreover, the induced map
$$ i^*\colon H^k(M(L);\mathbb{Z}) \longrightarrow H^k(M(J);\mathbb{Z})$$

\nd is an epimorphism which is an isomorphism for $k=2$.
\end{prop}
\begin{proof}
Diagram \eqref{cd:new.diagram} implies that the diagram below commutes up to homotopy.
(Notice here that the rows are not fibrations up to homotopy.)
\begin{equation*}\label{cd:normal.bundles}
\begin{CD}
M(J)  @>{i}>> 
ET^{m}\times_{T^{m}}  Z(K;(\underline{D}^{2J}, \underline{S}^{2J-1})) @>{p_1^J}>>  BT^m\\
@VV{\zeta}V           @VV{\overline{\zeta}}V     @VV{}V\\
M(L) @>{i}>> ET^{m}\times_{T^{m}}  
Z(K;(\underline{D}^{2L}, \underline{S}^{2L-1})) @>{p_1^L}>>  BT^m.
\end{CD}
\end{equation*}
\skp{0.1}
\nd Theorems \ref{thm:dj.analogue}  and  \ref{thm:cohom.dj.analog} imply that $\overline{\zeta}^*$
is onto and the Serre spectral sequence part of the proof of Theorem \ref{thm:top.calculation} shows
that $i^*$ is onto. This implies that $\zeta^*$ is onto too. The first part of the proposition follows from
the fact that each canonical bundle $L_i$ over $BT^m$ pulls back to the bundle $l_i\alpha_i$ 
over $M(L)$ and to $j_i\alpha_i$ over $M(J)$.\end{proof}

 A ``nest'' of toric manifolds is constructed below from a given  toric manifold triple 
$(P^n, M^{2n}, \lambda)$, an 
initial sequence $J_0 = (1,1,\ldots,1)$ with $m$ entries corresponding to the number of facets of $P^{2n}$ and 
a sequence $J_0 <  J_1 < J_2 < \cdots$ with the property that $d(J_{i+1}) = d(J_i) + 1$. (The symbol $d(J)$ is
defined at the end of Section \ref{sec:new.complex}).

\begin{prop}\label{thm:nests}
Given a toric manifold $(P^n, M^{2n}, \lambda)$ and a sequence $J_0 <  J_1 < J_2 < \cdots$ 
as above, there is a nest of toric manifolds,
$$M^{2n} = M(J_0) \subset M(J_1) \subset \cdots \subset M(J_k) \subset \ldots$$

\nd where $d(J_i) = m + i$, the dimension of $M(J_i)$ is $2n+2i$ and $M(J_i) \subset M(J_{i+1})$
is a codimension-two embedding. Furthermore, there is a sequence of epimorphisms
$$ \cdots \longrightarrow H^k(M(J_i);\mathbb{Z})  \longrightarrow H^k(M(J_{i-1});\mathbb{Z})
\longrightarrow \cdots \longrightarrow H^k(M^{2n};\mathbb{Z})$$

\nd which are isomorphisms for $k=2$. \hfill $\square$
\end{prop}

\nd {\bf Problem:\/} Beginning with a toric manifold $(P^n, M^{2n}, \lambda)$, find invariants,
possibly in terms of $\lambda$, which will detect diffeomorphic, (homotopic), manifolds in
nests corresponding to different sequences $J_0 < J_1 < J_2 <\cdots $.\\[5mm]
\nd \centerline{\LARGE{$\longrightarrow\!\longleftarrow$}}

\newpage
\bibliographystyle{amsalpha}

\end{document}